\documentclass[12pt,twoside,final]{amsart}
\usepackage[draft=false]{hyperref}
\usepackage[psamsfonts]{amssymb}
\usepackage{times,a4wide}

\usepackage{xcolor}

\theoremstyle{plain}
\newtheorem*{theorem*}{Theorem}

\newtheorem{theorem}{Theorem}
\newtheorem{proposition}[theorem]{Proposition}
\newtheorem*{maintheorem*}{Main Theorem}
\newtheorem*{proposition*}{Proposition}

\newtheorem*{corollary*}{Corollary}
\newtheorem{lemma}[theorem]{Lemma}
\newtheorem*{lemma*}{Lemma}

\theoremstyle{definition}
\newtheorem{remark}[theorem]{Remark}

\newtheorem*{remark*}{Remark}
\newtheorem*{remarks*}{Remarks}
\newtheorem*{conjecture*}{Conjecture}
\theoremstyle{definition}
\newtheorem*{definition}{Definition}

\allowdisplaybreaks

\usepackage[active]{srcltx}
\usepackage{tikz}
\usepackage{pgfplots}
\usepackage{float}
\RequirePackage{pgfplots}
\usetikzlibrary{shadows}
\pgfplotsset{compat=1.16}

\newcommand{\C}{\mathbb{C}}
\newcommand{\B}{\mathbb{B}}

\newcommand{\N}{\mathbb{N}}
\newcommand{\R}{\mathbb{R}}

\renewcommand{\P}{\mathcal{P}}

\newcommand{\supp}{\operatorname{Supp}}

\newcommand{\lilk}{\lambda\in\Lambda_k}

\title[Multiple sampling and interpolation  on the 
sphere] 
{Multiple sampling and interpolation  \\
in a space of polynomials} 

\author[C.A. Cruz, X. Massaneda, J. Ortega-Cerd\`a]{Carlos A. Cruz,
Xavier Massaneda, Joaquim Ortega-Cerd\`a}

\address{C.A. Cruz, X. Massaneda: Departament de Matem\`atiques i Inform\`atica,
Universitat  de Barcelona, Gran Via 585, 08007-Bar\-ce\-lo\-na, Spain}

\address{J. Ortega-Cerd\`a: Departament de Matem\`atiques i Inform\`atica,
Universitat  de Barcelona and CRM, Centre de Recerca Matemàtica, Bellaterra, Spain}

\thanks{Second and third authors partially supported by the Generalitat de Catalunya (grant 2021 SGR 00087) and the Spanish Ministerio de Ciencia e Innovaci\'on (project PID2021-123405NB-I00). Third author also supported by the Spanish State Research Agency, through the
Severo Ochoa and Mar\'{\i}a de Maeztu Program for Centers and Units of Excellence in
R\&D (CEX2020-001084-M).}

\date{\today}



\begin{document}

\begin{abstract} 
 We study sampling and interpolation arrays with multiplicities for the spaces 
$\P_k$ of holomorphic polynomials of degree at most $k$. 
We find that the geometric conditions satisfied by these arrays are in 
accordance with the conditions satisfied by the sampling and interpolating 
sequences with unbounded multiplicities in the Fock space, which can be seen as 
a limiting case of the space $\P_k$ as $k$ tends to infinity.
\end{abstract}

\maketitle

\section{Introduction and main results}
The space of homogeneous polynomials of degree $k$ in two variables 
of the form $Q(z,w) = \sum_{j = 0}^k a_j z^j w^{k-j}$
can be endowed with the norm given by
\begin{equation}\label{norma}
   \|Q\|^2_k = (k+1)\int_{\partial \mathbb B} |Q(z,w)|^2 \, d\mu,
\end{equation}
where $\mu$ is the surface measure on the unit ball $\mathbb B$ in 
two complex variables, normalized such that $\mu(\partial \mathbb B)=1$.
The monomials are orthogonal and an elementary computation, see 
\cite[Prop. 1.4.4]{Rudin}, shows that
$$
\|Q\|^2_k = \sum_{j= 0}^k \frac{|a_j|^2}{\binom{k}{j}}. 
$$
This is known as the Bombieri norm (or sometimes the Weyl norm) of an 
homogeneous polynomial. Remark that we have introduced the factor $(k+1)$ in 
\eqref{norma} so that $\|z^k\|_k = 1$.

To every homogeneous polynomial $Q$ of degree $k$ of two 
variables, we can associate a holomorphic polynomial in one complex 
variable $P(z) = Q(z,1) = \sum_{j = 0}^k a_j z^j$ of at most degree $k$ endowed 
with the same norm: $\|P\|^2_k = \sum_{j= 0}^k |a_j|^2/\binom{k}{j}$. 
We denote by $\P_k$ this space of polynomials. It is a reproducing kernel 
Hilbert space, when we endow it with the associated scalar product.

The most interesting geometric way of interpreting this norm is the following:
Consider any polynomial $p\in \P_k$ and, for every point $z\in \mathbb C$,
consider the weighted pointwise norm $s(z) = \frac{|p(z)|}{(1+|z|^2)^{k/2}}$. 
Then 
\[
   \|p\|_k^2 = (k+1) \int_{\mathbb S^2} s(\Pi(w))^2 d\Theta(w),
\]
where $\Pi$ is the stereographic projection from the sphere $\mathbb S^2$ in 
$\mathbb R^3$ to the complex plane and $\Theta$ is the probability measure in 
$\mathbb S^2$ invariant under rotations. The pushforward of $\Theta$ by $\Pi$ is 
given by
\[
 \nu =\Pi^*\Theta=\frac{dm(z)}{\pi(1+|z|^2)^2},\quad z\in\C.
\]
Here $dm$ denotes the Lebesgue measure in the plane.
Thus, the norm of any $p\in \P_k$ can be written directly in the complex plane as
\[
 \|p\|_{k,2}
 =\left((k+1)\int_{\C} \frac {|p(z)|^2}{(1+|z|^2)^k}\, d\nu(z)\right)^{1/2}.
\]
The normalizing constant $k+1$ ensures that $\|1\|_{k,2}=1$.

\subsection{On the unreasonable ubiquity of the space $\P_k$}
The space of polynomials $\P_k$ with its associated scalar product has 
appeared repeatedly in the literature in different contexts.

The fact that the group of rotations $SU(2)$ 
acts isometrically on the space of homogeneous polynomials in two variables 
endowed with the $L^2$ norm on the unit sphere of $\C^2$, entails that for each 
$k$ we have an irreducible representation of $SU(2)$. This makes this space 
interesting from the point of view of mathematical physics. In this context, the 
space $\P_k$ represents spin quantum states of spin $k/2$, see \cite{Bodmann}. 
The generalized Wherl entropy
conjecture proved by Lieb and Solovej in \cite{LiebSolovej} can be rephrased in 
terms of the polynomials in $\P_k$ as to whether the 
maximizer of a convex functional among all polynomials in $\P_k$ with norm one 
is attained exactly at the reproducing kernel. For a different solution, see 
\cite{Frank23} and \cite{KNOT22}. 

Another situation where this space of function arises is in the study of random 
point processes in the unit sphere in $\R^3$ with a distribution that is
invariant under rotations. Two processes that exhibit such invariance are naturally associated to $\P_k$. The first is obtained by taking any random element,
with uniform distribution in 
the unit sphere of $\P_k$, and map its complex zeros through the 
inverse of the stereographic projection to the sphere in $\R^3$. This is the 
spherical Gaussian analytic function (GAF) that is studied, for instance, in 
\cite[Chapter~2]{GAFbook}. The other point 
process is the so-called spherical ensemble. This is a point process that can 
be realized either by the generalized eigenvalues of some ensemble of random 
matrices or through a determinantal point process induced by the reproducing 
kernel of $\P_k$, as shown by Krishnapur in 
\cite{Krishnapur}. Both point 
processes, the spherical GAF and the spherical ensemble, have very good 
discrepancy estimates and are very interesting from the 
point of view of potential theory, since typically they provide points with very small 
logarithmic energy in the sphere. This was observed by Alishahi and Zamani in 
\cite{AliZam} for the spherical ensemble 
and by Armentano, Beltran and Schubb \cite{Beltran} for the spherical GAF. They 
are also very useful as nodes 
for Monte Carlo integration methods, as studied by Berman, see for instance 
\cite{Berman23}. 

Another example where $\P_k$ shows up is in the study of sharp constants for 
estimates from below on the norm of product of polynomials. The natural norm in 
$\P_k$ admits very precise estimates. For instance if $p\in \P_N$ factors as 
$p=p_1p_2$, where $p_j$ has degree $N_j$, with $N=N_1+N_2$ then Beauzamy, 
Bombieri, Enflo and Montgomery proved in \cite{BBEM} that
\begin{equation}
\frac{N_1!N_2!}{N!}\,\|p_1\|^2_{N_1}\|p_2\|^2_{N_2}\le\|p\|^2_{N}.
\label{eq:Bomb0}
\end{equation}
and the inequality is sharp. This is usually called the Bombieri inequality. 

Another problem, involving the Bombieri norm, was posed by Rudin: find
an orthonormal basis of homogeneous polynomials of 
degree $k$ in $L^2(\partial \B)$ with bounded uniform norm. The original motivation 
was to build an inner function in several complex variables. The basis
problem was solved by Bourgain in dimension two in \cite{Bourgain1} and in 
dimension three in \cite{Bourgain2}. In higher dimensions, it is still an open 
problem.  
In dimension two, the problem can be rephrased as a problem in the $\P_k$ space: for
each $k$, find an orthonormal basis of polynomials $p_i\in \P_k$ such that for all 
$i = 0,\ldots, k$, $\sup_z\frac{|p_i(z)|}{(1+|z|^2)^{k/2}}\le C$, where $C$ is 
independent of $k$.

In complex geometry, the spaces $\P_k$ are the toy model for the space of 
holomorphic sections of high powers of a positive line bundle over a compact 
manifold. 
The model example is the complex projective space $X = \mathbb{CP}^n$ 
with
the hyperplane bundle $L = \mathcal O(1)$, endowed with the Fubini-Study metric.
The $k$'th power of $L$ is denoted $\mathcal O(k)$, and the holomorphic 
sections
to $\mathcal O(k)$ can be identified with the homogeneous polynomials of degree 
at most $k$ in $n+1$ variables. This is exploited explicitly in \cite{LOC}, 
where sampling and interpolation arrays are studied for general sections of 
positive holomorphic line bundles, but more precise results are obtained when 
$n=1$ for the model case $\P_k$.

Another use of the space $\P_k$ is that as we increase the degree, 
the space looks closer to the Fock space, i.e., the space of entire functions that are in 
$L^2$ of a Gaussian weight. This is well-known in the Mathematical Physics 
literature, where one approximates Glower coherent states (corresponding to
reproducing kernels of the Fock space) by quantum spin states with high spin 
(corresponding to reproducing kernels in $\P_k$), see \cite{Bodmann} where 
this is explicitly stated. The hope is that by working in a compact setting 
(our functions are essentially defined in the sphere) we avoid some 
technical difficulties, and we can pass to the limit and get results transferred from $\P_k$ 
to the Fock space, see for instance \cite{LOC}. This hope has not been 
realized. Many of the results that are proved in the Fock space and in $\P_k$ 
have been obtained first in the non-compact setting, where the invariant under 
translations seems to be more useful.  

The ubiquity of this space of polynomials makes them very attractive to study 
classical function space problems on them, like  norming sets, 
interpolating sequences, sampling sequences, Toeplitz operators, 
etc.

In this paper we will concentrate on sampling and interpolating arrays, that is 
the discretization of the space of functions through a restriction to a finite 
set (the precise definitions will be given below). The basic case was discussed 
in \cite{LOC} where a description of sampling and interpolation arrays are 
given. In this paper we will introduce multiplicities, that is, we are 
interested in sampling and interpolating not just the values of the polynomial 
but also the derivatives. We will consider unbounded multiplicities, that is, 
we will allow to interpolate/sample more and more derivatives as the degree of 
$\P_k$ increases. Our guiding principle is the paper by 
Borichev, Hartmann, Kellay and Massaneda \cite{BHKM}, where the corresponding 
problem for the Fock space is considered.

\subsection{The geometry}

The \emph{chordal distance} in $\mathbb S^2$, i.e. the Euclidean distance in $\R^3$ restricted to $\mathbb S^2$,  is projected on the plane to the distance
\[
 d(z,w)=\frac{|z-w|}{(1+|z|^2)^{1/2}(1+|w|^2)^{1/2}}, \quad z,w\in\C.
\]
Observe that $d$ is normalized so that the distance between two diametral opposite points is $1$. 

Given $z\in\C$ and $r>1$ we denote by $D(z,r)$ the chordal disc centered at $z$ with radius $r$, i.e.
\[
 D(z,r)=\bigl\{w\in\C : d(z,w)<r\bigr\}.
\]
These discs correspond to spherical caps in $\mathbb S^2$.

Notice that, with the previous notation, 
\[
\|p\|_{k,2}^2=(k+1)\int_{\C} \bigl|p(z) e^{-\frac k2\phi(z)}\bigr|^2\, d\nu(z),
\]
hence $s(z):=p(z) e^{-\frac k2\phi(z)}$ can be seen as a section of a power line bundle over $\mathbb S^2$ with metric induced by $\phi$.

The space $ \mathcal P_k$ is a reproducing kernel Hilbert space with the natural inner product
\[
 \langle p, q\rangle_k=(k+1)\int_{\C}p(z)\, \overline{q(z)}\, \frac{d\nu(z)}{(1+|z|^2)^k},\quad p,q\in\P_{k}.
\]
Thus, for each $w\in\C$ there exists $K_w\in\P_k$ such that 
\[
  \langle p, K_w\rangle_k=p(w)\quad \forall p\in\P_{k}.
\]
It is easy to check that the system
 \begin{equation}\label{ob}
  e_{k,j}(z):=\binom{k}{j}^{1/2} z^j,\quad j=0,\dots,k
 \end{equation}
 is an orthonormal basis of $\P_{k}$. Therefore, the reproducing kernel $K_w$ has the form
 \[
  K_w(z)=K(z,w)=\sum_{j=0}^k \overline{e_{k,j}(w)}\, e_{k,j}(z)=(1+z\bar w)^k.
 \]

\subsection{Sampling and interpolation arrays}
In order to define multiple sampling and interpolation for the family $\{\mathcal P_k\}_{k\geq 1}$ we need to introduce the isometries of $\P_k$.

Both the measure $\nu$ and the chordal distance are obviously invariant by 
the rotations of the sphere, which in the plane take the form of the M\"obius transformations
\begin{equation*}
 \varphi_{\lambda,\theta}(z):=e^{i\theta}\frac{\lambda-z}{1+\bar\lambda z},\quad \lambda,z\in\C,\ \theta\in[0,2\pi).
\end{equation*}
In particular, we denote 
 \begin{equation}\label{eq:automorphism}
 \varphi_\lambda(z)=\varphi_{\lambda,0}(z)=\frac{\lambda-z}{1+\bar\lambda z}.
 \end{equation}
This corresponds to a rotation of $\, \mathbb S^2$ exchanging $\Pi^{-1}(\lambda)$ and $\Pi^{-1}(0)=0$ (where, as before, $\Pi$ indicates the stereographic projection). 

Given $p\in\P_k$ and $\lambda\in\C$ define
 \[
  T_\lambda\, p(z)=\left[\frac{(1+\bar\lambda z)^2}{1+|\lambda|^2}\right]^{k/2}\, p\bigl(\varphi_\lambda(z)\bigr),\quad z\in\C.
 \]
 A straightforward computation, using the identity
  \begin{equation}\label{eq:one-plus}
   1+|\varphi_\lambda(z)|^2=\frac{(1+|\lambda|^2)(1+|z|^2)}{|1+\bar\lambda z|^2},\quad \lambda,z\in\C
  \end{equation}
  and the invariance of $\nu$,
  shows that $T_\lambda$ is an isometry in $\P_k$ with respect to the norm $\|\cdot\|_{k,2}$.

 For each $k\geq 1$ consider a finite set $\Gamma_k=\{\zeta_1^{(k)},\dots, \zeta_{j_k}^{(k)}\}\subset \mathbb S^2$, which is projected to a set $\Lambda_k=\{\lambda_1^{(k)},\dots, \lambda_{j_k}^{(k)}\}$ in $\C$ ($\Pi(\zeta_{j}^{(k)})=\lambda_{j}^{(k)}$).  For each $\lambda\in\Lambda_k$ let $m_\lambda\in \N$ be an associated multiplicity. This gives rise to the vector of multiplicities
 $m^{(k)}=(m_{\lambda_1^{(k)}},\dots, m_{\lambda_{j_k}^{(k)}})\in \N^{j_k}$ associated to $\Lambda_k$.

\begin{definition}
 An array $X=\{(\Lambda_k, m^{(k)}\}_{k\geq 1}$ is \emph{$L^2$-sampling} if there exists $C>0$ independent of $k$ such that for all $p\in\mathcal P_k^2$
 \[
  \frac 1C \|p\|_{k,2}^2\leq \sum_{\lambda\in\Lambda_k}\sum_{j=0}^{m_\lambda-1} \left|\langle p,T_{\lambda} e_{k,j}\rangle_k\right|^2\leq C\|p\|_{k,2}^2.
 \]
 Alternatively, $X$ is a $L^2$-sampling array if $\{T_\lambda e_{k,j}\}_{\lambda\in\Lambda_k, j<m_\lambda}$ is a frame for $\mathcal P_k$, with frame constants independent of $k$.
\end{definition}

Less formally, an array $X$ is $L^2$-sampling when the values $p(\lambda), \dots, p^{(m_\lambda-1)}(\lambda)$, $\lambda\in\Lambda_k$, determine, with uniform control of norms, any polynomial $p\in\P_k$. Since $\mathcal P_k$ is finite dimensional, the whole point in this definition is that the constant $C$ is independent of $k$.

Given an array $X=\{(\Lambda_k, m^{(k)}\}_{k\geq 1}$ consider also associated arrays of values $v_X=\{v_X^k\}_{k\geq 1}$ in the following way: to each $\lambda\in\Lambda_k$, $k\geq 1$, assign a vector of values 
$v_\lambda=(v_{\lambda}^j)_{j<m_\lambda}$ and denote $v_X^k=\{v_\lambda\}_{\lambda\in\Lambda_k}$.

Consider also the $\ell^2$ norm at each level $k$:
\[
 \|v_X^k\|_{\ell^2}=\left(\sum_{\lambda\in\Lambda_k} \|v_\lambda\|_2^2\right)^{1/2}=
 \left(\sum_{\lambda\in\Lambda_k} \sum_{j<m_\lambda} |v_\lambda^j|^2\right)^{1/2}.
\]

 \begin{definition} 
  An array $X=\{(\Lambda_k, m^{(k)}\}_{k\geq 1}$ is \emph{$L^2$-interpolating} if there exists $C>0$ independent of $k$ such that for all $k\geq 1$ and all arrays of values $v_X=\{v_X^k\}_{k\geq 1}$ there exist polynomials $p_k\in\mathcal P_k$ with:
  \begin{itemize}
   \item [(a)] $\langle p_k, T_\lambda e_{k,j}\rangle_k=v_\lambda^j$ for all $\lambda\in\Lambda_k$ and $j<m_{\lambda}$,
   \item [(b)] $\|p_k\|_{2,k}\leq C \|v_X^k\|_{\ell^2}$.
  \end{itemize}
  The minimum such $C$ is called the interpolation constant of $X$, and it will be denoted by $M_X$.
 \end{definition}
 
 Thus, an array $X$ is interpolating when it is possible to prescribe the value and the derivatives up to order $m_\lambda-1$ on all $\lambda\in\Lambda_k$, $k\geq 1$.  At a fixed level $k$, if $\sum_{\lambda\in\Lambda_k} m_\lambda\leq k+1=\text{dim}(\P_k)$ one can always find polynomials $p$ satisfying the system of equations (a). The point in this definition, again, is that this can be done with the uniform control of norms given by (b), independently of $k$.
 
 \begin{remark}
  Sampling and interpolation arrays can be equivalently defined in a more intrinsic way (easier to generalize to other norms). Let $X$ be as before. For each $\lambda\in\Lambda_k$, $k\geq 1$, consider the vanishing subspace of $\P_k$ on $\lambda$:
  \[
N_{\lambda,m_\lambda}=\Bigl\{p\in\P_k : p^{(j)}(\lambda)=\frac{\partial^j p}{\partial z^j}(\lambda)=0, j<m_\lambda\Bigr\}.
  \]
Then $X$ is $L^2$-sampling if there exists $C>0$ such that for all $k \geq 1$ and all $p_k\in \P_k$
 \[
  \frac 1C \|p_k\|_{k,2}^2\leq \sum_{\lambda\in\Lambda_k}\|p_k\|_{\P_k/N_{\lambda,m_\lambda}}^2\leq C\|p_k\|_{k,2}^2.
 \]
 
 Similarly, an array $X$ is $L^2$-interpolating if there exists $C>0$ such that for every $k$ and every vector $(p_\lambda)_{\lambda\in\Lambda_k}\subset\P_k$ there exist $p_k\in\P_k$ such that:
 \begin{itemize}
  \item  [(a)] $p_k-p_\lambda\in N_{\lambda, m_\lambda}$ for all $\lambda\in\lambda_k$,
  \item [(b)] $\|p_k\|_{2,k}^2\le C\sum\limits_{\lambda\in\Lambda_k} \|p_k\|_{\P_k/N_{\lambda,m_\lambda}}^2$.
 \end{itemize}
 \end{remark}
 
 Since $\dim(\P_K)=k+1$ it is clear that the conditions $\sum_{\lambda\in\Lambda_k} m_\lambda\leq k$, for all $k\geq 1$, are necessary.

 \begin{remark}
  A zero set in $\P_k$ has at most $k$ points, so a uniqueness set must have at least $k$ points. In particular, in a sampling array necessarily $\sum_{\lambda\in\Lambda_k} m_\lambda>k$ for all $k\geq 1$ and in an interpolation array $\sum_{\lambda\in\Lambda_k} m_\lambda\leq k$ for all $k\geq 1$. Optimal sampling or interpolation arrays should have $\sum_{\lambda\in\Lambda_k} m_\lambda$ of order $k$ as $k\to\infty$. 
 \end{remark}

 \subsection{Results}
 In order to state the geometric conditions satisfied by sampling or interpolation arrays, we need some definitions.
 
 At each level $k$, we shall consider separation conditions in terms of the rescaled distances $\sqrt k\, d(z,w)$. In particular, for $\lambda\in\Lambda_k$ we shall consider the ``critical'' discs
  \begin{equation}\label{def:critical-disks}
  D_{\lambda,m_\lambda}:=D\Bigl(\lambda, \sqrt{\frac{m_\lambda}k}\Bigr)=\bigl\{z\in\C : \sqrt k\, d(z,\lambda)<\sqrt{m_\lambda}\bigr\}
 \end{equation}

As we shall see, loosely speaking, prescribing the value and the derivatives up to order $m_\lambda$ of a polynomial $p\in\P_k$ at a point $\lambda$  determines its growth in a chordal disc of radius $\sqrt{m_\lambda/k}$.

 \begin{definition}
  An array $X=\{(\Lambda_k, m^{(k)})\}_{k\geq 1}$ has \emph{finite overlap} if for every $z\in\C$ there is at most a fixed number (independent of $k$) of discs $D_{\lambda,m_\lambda}$, $\lambda\in\Lambda_k$, containing $z$, i.e., if
  \[
   S_X:=\sup_{k\geq 1}\sup_{z\in\C} \sum_{\lambda\in\Lambda_k} \chi_{D_{\lambda,m_\lambda}}(z)<\infty.
  \]
 \end{definition}
 
 Since $\nu\bigl(D_{\lambda,m_\lambda}\bigr)=m_\lambda/k$ (see Lemma~\ref{lemma:computations}(a)), the finite overlap condition implies 
 \[
  \sum_{\lambda\in\Lambda_k}\frac{m_\lambda}k=\sum_{\lambda\in\Lambda_k} \nu\bigl(D_{\lambda,m_\lambda}\bigr)
  =\sum_{\lambda\in\Lambda_k}\int_{\C}\chi_{D_{\lambda,m_\lambda}}(z)\, d\nu(z)\leq S_X.
 \]

Henceforth, we shall assume that
 \begin{equation}\label{eq:m-to-zero}
  \lim_{k\to\infty}\frac{\sup\limits_{\lambda\in\Lambda_k} m_\lambda}k=0.
 \end{equation}

 Our conditions for sampling or interpolation are expressed in terms of constant dilations or contractions of the critical disks $D_{\lambda,m_\lambda}$, $\lambda\in\Lambda_k$. 
 Given $\lambda\in\C$, $k\geq 1$ and $m<k$ consider $c>-\sqrt{m}$ and let 
  \[
 D\bigl(\lambda,\frac{\sqrt m +c}{\sqrt k}\bigr)= \{z\in\C : \sqrt k \, d(z,\lambda)\leq \sqrt m +c\}.
 \]
Observe that for $c=0$ we have the critical disk $D_{\lambda,m_\lambda}$.

 We begin with conditions for multiple interpolation. 
  
 \begin{theorem}\label{thm:intL2}
  (a) If $X=\{(\Lambda_k,m^{(k)})\}_{k\in\N}$ is an $L^2$-interpolation array with interpolation constant $M_X$, then there exist $c=c(M_X)>0$ and $k_0\in\N$ such that for $k\geq k_0$ the contracted disks 
  \[
  \Bigl\{D\bigl(\lambda,\frac{\sqrt m_\lambda -c}{\sqrt k}\bigr)\Bigr\}_{\lambda\in\Lambda_k : \sqrt{m_\lambda}>c}
  \]
  are pairwise disjoint.
  
  (b) If there exist $c>0$ and $k_0\in\N$ such that for $k\geq k_0$ the dilated disks
   \[
  \Bigl\{D\bigl(\lambda,\frac{\sqrt m_\lambda +c}{\sqrt k}\bigr)\Bigr\}_{\lambda\in\Lambda_k}
  \]
  are pairwise disjoint, then $X$ is an $L^2$-interpolating array for $\P_k$.
 \end{theorem}

 The next result gives necessary conditions for sampling. 
 
 \begin{theorem}\label{th:sampling-L2}
 Let $X=\{(\Lambda_k, m^{(k)})\}_{k\geq 1}$ be an $L^2$-sampling array for $\P_k$. Then:
 \begin{itemize}
  \item [(a)] $X$ satisfies the finite overlap condition,
  \item [(b)] There exist $c>0$ and $k_0\in\N$ such that for all $k\geq k_0$
  \[
   \bigcup_{\lambda\in\Lambda_k} D\bigl(\lambda,\frac{\sqrt m_\lambda +c}{\sqrt k}\bigr)=\C.
  \]
 \end{itemize}
 \end{theorem}

 The sufficient condition for sampling is expressed in similar terms.
 
 \begin{theorem}\label{thm:L2-sampling-suff}
 Let $X$ be a divisor satisfying the finite overlap condition. For $c>0$ and $k\in\N$ let the union of contracted discs
 \[
 \mathcal U_k(-c)=\bigcup_{
 \lilk : \sqrt {m_\lambda}>c} D\bigl(\lambda,\frac{\sqrt m_\lambda -c}{\sqrt k}\bigr).
 \]
 There exists $c=c(S_X)>0$ such that if
 \[
 \lim_{k\to\infty} k\, \nu\bigl(\C\setminus \mathcal U_k(-c)\bigr)=0,
 \]
 then $X$ is $L^2$-sampling array.
\end{theorem}

\begin{remarks*} 1. Our study is motivated by the analogous results for the Fock space of entire functions
\[
 \mathcal F=\bigl\{f\in H(\C) : \|f\|_{\mathcal F}^2=\frac 1{\pi}\int_{\C} |f(z)|^2 e^{-|z|^2}dm(z)<\infty\bigr\},
\]
obtained in \cite{BHKM}. The theorems above are in accordance with the well-known fact that any $f\in\mathcal F$ can be seen as the limit of a sequence of rescaled polynomials $p_k(z/\sqrt k)$, $k\in\N$, with $p_k\in\P_k$ (see e.g. \cite{Gr-OC}). Actually, Theorem 1.1 in \cite{BHKM} can be seen as the limit case of Theorems~\ref{th:sampling-L2} and \ref{thm:L2-sampling-suff}, whereas Theorem 1.3 in \cite{BHKM} is the limit case of Theorem~\ref{thm:intL2}.

Sampling and interpolating sequences with unbounded multiplicities have also been studied for the Bergman spaces in the unit disk (see \cite{ACHK}).

2. It would be interesting to study interpolating and sampling arrays on general compact complex manifolds $X$, endowed with a smooth hermitian metric $\omega$ inducing a distance $d$. Starting from a holomorphic line bundle $L$ equipped with a smooth hermitian metric $\phi$, the role of the spaces $\P_k$ would be played by the $k$-th tensor power $L^k$, with metric $k\phi$ (see \cite{LOC} for similar problems).

\end{remarks*}

The paper is organized as follows. In the first section, we set up the elementary technical results necessary for the proofs. Section~\ref{sec:local} deals with local estimates of polynomials $p\in\P_k$. These are crucial in the proofs of the theorems stated in this introduction, and rely heavily on careful estimates of the incomplete beta functions. The proof of these estimates is deferred to an Annex, at the end of the paper. Section~\ref{sec:L2-sampling} is devoted to prove the sampling Theorems~\ref{th:sampling-L2} and \ref{thm:L2-sampling-suff}. Finally, Section~\ref{sec:int} contains the proof of Theorem~\ref{thm:intL2}, which is an adaptation of the ideas in \cite{BOC}.

 \section{Preliminaries}
 
Recall that $\varphi_\lambda$ is the rotation of the sphere given by \eqref{eq:automorphism}. It will be convenient to use the disks
 \[
  \Delta(z,r)=\bigl\{w\in\C : |\varphi_z(w)|<r\bigr\},\quad z\in\C, \ r>0.
 \]
 Using \eqref{eq:one-plus} one sees that for $z,\lambda\in\C$,
 \[
  d^2(z,\lambda)=\frac{|\varphi_\lambda(z)|^2}{1+|\varphi_\lambda(z)|^2}.
 \]
 Therefore $\Delta(z,r)=D(z,\frac r{\sqrt{1+r^2}})$ and, in particular,
 \[
  D_{\lambda,m_\lambda}=\Delta(\lambda,r_{\lambda, m_\lambda})\ ,\quad \text{with} \quad r_{\lambda, m_\lambda}=\sqrt{\frac{m_\lambda/k}{1-m_\lambda/k}} =\sqrt{\frac{m_\lambda}{k-m_\lambda}}.
 \]

 We shall also consider the so-called invariant Laplacian. For $\psi\in\mathcal C^2$ denote 
 \[
 \Delta\psi=\frac i{2\pi}\partial\bar\partial\psi=\frac{\partial^2 \psi}{\partial z\partial \bar z}\frac{dm(z)}{\pi}, 
 \]
 which differs from the usual definition by a constant factor. As observed in the first paragraph,
 \[
  \Delta\phi=\nu.
 \]
Define 
 \[
  \widetilde \Delta =(1+|z|^2)^2\Delta.
 \]
 This operator is invariant with respect to the transformations $\varphi_\lambda$, $\lambda\in\C$, in the sense that
 \[
  \widetilde \Delta ( u\circ\varphi_\lambda )=(\widetilde  \Delta u )\circ \varphi_\lambda.
 \]
 In particular
 \begin{equation}\label{inv-lap}
\widetilde \Delta\log (1+|\varphi_\lambda(z)|^2)= \widetilde \Delta\log (1+|z|^2)=\frac{dm(z)}{\pi}.
 \end{equation}
 
 For future use we state here an elementary computation. 

 \begin{lemma}\label{lemma:computations}
  Let $\lambda\in\C$ and $R>0$. Then
\[
 \int_{D(\lambda,R)} \Delta\log(1+|\varphi_\lambda(z)|^2)=\nu\bigl(D(\lambda,R)\bigr)=   R^2.
\]
 \end{lemma}
 
Observe that this can be rephrased as
 \[
  \int_{\Delta(\lambda,r)} \Delta\log(1+|\varphi_\lambda(z)|^2)=\nu\bigl(\Delta(\lambda,r)\bigr)=  \frac { r^2}{1+r^2}, \quad r>0.
 \]
 
A particular case to be used throughout the paper is
 \[
  \nu(D_{\lambda,m_\lambda})=\nu\bigl(\Delta(\lambda, r_{\lambda,m_\lambda})\bigr)=\frac{m_\lambda}k.
 \]

\begin{proof}
 Using the invariance of $\nu$,
  \begin{align*}
   \int_{\Delta(\lambda,r)}  \Delta\log(1+|\varphi_\lambda(z)|^2)&= \int_{\Delta(\lambda,r)}  \widetilde\Delta\log(1+|\varphi_\lambda(z)|^2)\, \frac 1{(1+|z|^2)^2}=\nu \bigl(  \Delta(0,r)\bigr).
  \end{align*}
This is easily computed with a change to polar coordinates:
\begin{align*}
 \nu \bigl(  \Delta(0,r)\bigr)&=2\pi\int_0^r\frac{\rho\, d\rho}{\pi(1+\rho^2)^2}= \int_0^{r^2} \frac {dt}{(1+t)^2} = \Bigl[\frac 1{1+t}\Bigr]_0^{r^2}=\frac {r^2}{1+r^2} .
\end{align*}
\end{proof}

 \section{Local $L^2$ and $L^\infty$-estimates}\label{sec:local}
 
 We want to prove, quantitatively,  that the values of $p\in\P_k$ and its derivatives up to order $m$ at $\lambda$ are controlled by the local $L^2$ norm of $p$ in a disc of center $\lambda$ and radius of order $\sqrt{m/k}$. This requires precise estimates on the incomplete beta function.
 
 \subsection{Beta functions}\label{subsec:beta}
Given $a,b>1$ consider the usual beta functions of parameters $a$ and $b$: 
\[
 \beta(a,b)=\int_0^1 t^{a-1} (1-t)^{b-1} dt=\frac{\Gamma(a)\Gamma(b)}{\Gamma(a+b)}.
\]
The normalized incomplete beta functions appear naturally in our study. For $x\in (0,1)$ let
\[
 I(a,b; x)=\frac 1{\beta(a,b)}\int_0^x t^{a-1} (1-t)^{b-1} dt.
\]

For $R\in (0,1)$ and $p\in\P_k$ define the scalar product restricted to the disk $D(0,R)$,
 \begin{equation}\label{eq:scalar_R}
   \langle p,q\rangle_{k,R}=(k+1)\int_{D(0,R)} p(z)\, \overline{q(z)}\, \frac{d\nu(z)}{(1+|z|^2)^k},\quad p,q\in\P_k
\end{equation}
and its corresponding norm $\|p\|_{k,D(0,R)}=\bigl(\langle p,p\rangle_{k,R}\bigr)^{1/2}$.

It is clear that the system $e_{k,j}(z)=\binom{k}{j}\, z^j$, $j\in \{0,1,\dots,k\}$ is also orthogonal when restricted to $D(0,R)$. A bit more can be said.

 \begin{lemma}\label{lemma:beta-truncated} Let $\{e_{k,j}\}_{j=0}^k$ and $R>0$ as above. Then, for all $j\in \{0,1,\dots,k\}$,
  \[
   \|e_{k,j}\|_{k,D(0,R)}^2= I \bigl(j+1,k-j+1; R^2\bigr).
  \]
 \end{lemma}
 
 \begin{proof}
  By definition, since $z\in D(0,R)$ is equivalent to $|z|\leq\sqrt{R^2/(1-R^2)}$,
  \begin{align*}
   \|e_{k,j}\|_{k,D(0,R)}^2&=(k+1)\int\limits_{|z|\leq \sqrt{\frac{R^2}{1-R^2}}}
   \binom{k}{j} \frac{|z|^{2j}}{(1+|z|^2)^k}\,\frac{dm(z)}{\pi(1+|z|^2)^2}\\
  &=(k+1)\binom{k}{j} \int_0^{\sqrt{\frac{R^2}{1-R^2}}}\frac{r^{2j}}{(1+r^2)^{k+2}}\, 2r\, dr\\
  &=(k+1)\binom{k}{j} \int_0^{\frac{R^2}{1-R^2}}\frac{t^{j}}{(1+t)^{k+2}}\, dt.
  \end{align*}
On the one hand,
\[
 \beta(j+1,k-j+1)=\frac{j! (k-j)!}{(k+1)!}=\frac 1{(k+1) \binom{k}{j}}
\]
On the other hand, by the substitution $s=1/(1+t)$ in the integral,
\begin{align*}
 \int_0^{\frac{R^2}{1-R^2}}\frac{t^{j}}{(1+t)^{k+2}}\, dt&=
     \int_{1-R^2}^1 s^{k-j}(1-s)^j\, ds
   = \int_0^{R^2} x^j\, (1-x)^{k-j}\,  dx,
\end{align*}
as desired
 \end{proof}
 
 The following estimates for the incomplete beta function will be crucial. The proof is technical, and it is deferred to the Annex.

  \begin{lemma}\label{lemma:incomplete-beta} 
 For each $k\in\N$ let $m_k\in\N$ be such that $\lim_k m_k/k=0$.
  \begin{itemize}
   \item [(a)] For $a>0$ there exist $\epsilon(a)>0$ and $k_0,m_0\in\N$ such that for $k\geq k_0$ and $m\in[ m_0, m_k]$
   \[
    I\Bigl(j+1,k-j+1; \frac {m-a\sqrt m}k \Bigr)\geq \epsilon(a),\quad j=0,\dots,m-1.
   \]
   \item[(b)] Given $\epsilon>0$ there exist $a>0$, $k_0\in\N$, $m_0=m_0(a)>0$ such that for $k\geq k_0$ and $m\in[ m_0, m_k]$
\[
 I\Bigl(j+1,k-j+1; \frac {m-a\sqrt m}k \Bigr)\leq \epsilon\, I\Bigl(j+1,k-j+1;\frac mk\Bigr),\quad \text{for}\ j=m,\dots,k.
\]
  \end{itemize}
 \end{lemma}

 
 \subsection{Local estimates}
 
 We start with a straightforward pointwise estimate.

 \begin{lemma}\label{lemma:pointwise}
  Let $p\in\P_k$. Then
  \[
  \|p\|_{k,\infty}:=\sup_{\lambda\in\C} \frac{|p(\lambda)|}{(1+|\lambda|^2)^{k/2}}\leq \|p\|_{k,2}\, .
  \]
 \end{lemma}
 
  \begin{proof}
  Assume first $\lambda=0$. By the subharmonicity of $|p|^2$, for all $r>0$,
  \begin{align*}
   \int_{\Delta(0,r)}\frac{|p(z)|^2}{(1+|z|^2)^k}\, d\nu(z)&=\int_0^r\Bigl(\frac 1{2\pi} \int_0^{2\pi} |p(se^{i\theta})|^2 d\theta\Bigr)\, \frac{2s\, ds}{(1+s^2)^{k+2}}\\
   &\geq |p(0)|^2\, \int_0^r \frac{2s\, ds}{(1+s^2)^{k+2}}
   =\frac{|p(0)|^2}{k+1}\Bigl[1-\frac 1{(1+r^2)^{k+1}}\Bigr].
  \end{align*}
Therefore
\[
 |p(0)|^2\leq \Bigl[1-\frac 1{(1+r^2)^{k+1}}\Bigr]^{-1} (k+1)\,  \int_{\Delta(0,r)}\frac{|p(z)|^2}{(1+|z|^2)^k}\, d\nu(z),
\]
and letting $r\to\infty$ we get the desired estimate (for $\lambda=0$).

For the general case $\lambda\in\C$ we apply the previous estimate to $T_\lambda p$: by the norm-invariance
\[
 |\bigl(T_\lambda p\bigr)(0)|^2\leq \|T_\lambda p\|_{k,2}^2=\|p\|_{k,2}^2.
\]
Since
\[
 \bigl(T_\lambda p\bigr)(0)=\left(\frac {(1+0)^2}{1+|\lambda|^2}\right)^{k/2} p\bigl(\varphi_\lambda(0)\bigr)=\frac{p(\lambda)}{(1+|\lambda|^2)^{k/2}}
\]
the statement is proved.
 \end{proof}
 
 Next we prove local $L^2$-estimates. As always, assume that $m_k$ is a sequence in $\N$ with $\lim_k \frac{m_k}k=0$.
 
 \begin{lemma}\label{lemma:local-L2}
  For every $a>0$ there exist $C(a)>0$ and $m(a),k_0\in\N$ such that for all $k\geq k_0$, all $p\in\P_k$, all $m\in[ m(a),m_k]$ and all $\lambda\in\C$,
  \[
   \sum_{j=0}^{m-1}\bigl|\langle p, T_\lambda e_{k,j}\rangle_k\bigr|^2 \leq C(a) (k+1)
   \int_{D(\lambda, \frac{\sqrt m-a}{\sqrt k})}\frac{|p(z)|^2}{(1+|z|^2)^k}\, d\nu(z) .
  \]
 \end{lemma}

 \begin{proof} 
Taking $q=T_\lambda p$ instead of $p$ (if necessary), we can assume that $\lambda=0$. Then the inequality to be proved is
 \[
   \sum_{j=0}^{m-1}\bigl|\langle p, e_j\rangle_k\bigr|^2\leq C(a) (k+1)\int_{D(\lambda, \frac{\sqrt m-a}k)}\frac{|p(z)|^2}{(1+|z|^2)^k}\, d\nu(z) .
  \]
  Write $p=\sum_{j=0}^k a_j e_{k,j}$, where $e_{k,j}$ is the orthonormal basis given in \eqref{ob}.  Since the system $\{e_{k,j}\}_{j=0}^k$ is an orthogonal basis in any disc $D(0,R)$, with the restricted scalar product \eqref{eq:scalar_R}, by Lemma~\ref{lemma:beta-truncated}, the above estimate is just
  \begin{align*}
   \sum_{j=0}^{m-1} |a_j|^2&\leq C(a)  \sum_{j=0}^{k} |a_j|^2\ \|e_{k,j}\|_{k,D(\lambda, \frac{\sqrt m-a}k)}^2 \\
   &=C(a)  \sum_{j=0}^{k} |a_j|^2 I \Bigl(j+1,k-j+1;(\frac{\sqrt m-a}k)^2\Bigr).
  \end{align*}
  Since for $m\geq a^2$,
\[
m-2a\sqrt m\leq  (\sqrt m-a)^2\leq m-a\sqrt m,
\]
the result is the given by Lemma~\ref{lemma:incomplete-beta} (a).
 \end{proof}

  \begin{lemma}\label{lemma:local-L-inf}
  Given $\eta\in (0,1]$ there exist $a=a(\eta)>0$ and $m(a),k_0\in\N$  such that for all $k\geq k_0$, all $p\in\P_k$, all $m\in [m(a),m_k]$ and all $\lambda\in\C$, if
  \begin{align*}
   &\text{(i)}\quad \sum_{j=0}^{m-1}\bigl|\langle p,  T_\lambda e_{k,j}\rangle_k\bigr|^2 \leq \eta/2 , \\
   &\text{(ii)}\quad (k+1)\int_{D(\lambda, \sqrt{\frac mk})}\frac{|p(z)|^2}{(1+|z|^2)^k}\, d\nu(z) \leq 1.
  \end{align*}
then
  \[
  (k+1)
   \int_{D(\lambda, \frac{\sqrt{m}-a}{\sqrt k})} \frac{|p(z)|^2}{(1+|z|^2)^k}\, d\nu(z)\leq \eta.
  \]
 \end{lemma}

 \begin{proof} We proceed similarly to the proof of Lemma~\ref{lemma:local-L2}: assume that $\lambda=0$ and
 write $p=\sum_{j=0}^k a_j e_{k,j}$, so that, by Lemma~\ref{lemma:beta-truncated}, hypothesis (ii) is
 \begin{align*}
  (k+1)\int_{D(0, \sqrt{\frac mk})}\frac{|p(z)|^2}{(1+|z|^2)^k}\, d\nu(z)=\sum_{j=0}^k |a_j|^2\ I\Bigl(j+1,k-j+1; \frac mk\Bigr)\leq 1.
 \end{align*}
Similarly, by (i) we deduce that
\begin{align*}
 (k+1)\int_{D(\lambda, \frac{\sqrt{m}-a}{\sqrt k})}\frac{|p(z)|^2}{(1+|z|^2)^k}\, d\nu(z)&=
 \sum_{j=0}^k |a_j|^2\ I\Bigl(j+1,k-j+1; (\frac{\sqrt{m}-a}{\sqrt k})^2\Bigr)\\
 &\leq \sum_{j=0}^{m-1} |a_j|^2+ \sum_{j=m}^k |a_j|^2\ I\Bigl(j+1,k-j+1; \frac{m-a\sqrt m}{k}\Bigr)\\
& \leq \frac{\eta}2 + \sum_{j=m}^k |a_j|^2\ I\Bigl(j+1,k-j+1; \frac{m-a\sqrt m}{k}\Bigr).
\end{align*}
Therefore, it is enough to see that there exists $a>0$ and  $k_0, m(a)\in\N$ such that for $k\geq k_0$ and $m\in[m_0,m_k]$
\[
 \sum_{j=m}^k |a_j|^2\ I\Bigl(j+1,k-j+1;\frac{m-a\sqrt m}{k}\Bigr)\leq \frac {\eta} 2 \sum_{j=0}^k |a_j|^2\ I\Bigl(j+1,k-j+1;\frac mk\Bigr).
\]
This is precisely Lemma~\ref{lemma:incomplete-beta}(b).
 \end{proof}

\section{Proof of Theorems~\ref{th:sampling-L2} and ~\ref{thm:L2-sampling-suff} }\label{sec:L2-sampling}
Let's begin by showing that one of the estimates in the definition of $L^2$-sampling array is equivalent to the finite overlap condition.

\begin{proposition}\label{prop:overlap_sampling}
 An array $X=\{(\Lambda_k, m^{(k)})\}_{k\geq 1}$  has finite overlap if and only if there exist $C>0$ and $k_0\in\N$ such that for all $k\geq k_0$ and all $p\in\mathcal P_k$
 \[
  \sum_{\lambda\in\Lambda_k}\sum_{j=0}^{m_\lambda-1} \bigl|\langle p, T_\lambda e_{k,j}\rangle_k\bigr|^2\leq C \|p\|_k^2.
 \]
\end{proposition}

 \begin{proof}
  Assume first that the inequality in the statement holds. Fix $z\in\C$ and define $p=T_z 1=T_z e_{k,0}$. Then
  \[
   \langle p, T_\lambda e_{k,j}\rangle_k=\langle T_z e_{k,0}, T_\lambda e_{k,j}\rangle_k
   =\langle  e_{k,0}, T_z(T_\lambda e_{k,j})\rangle_k=\overline{T_z(T_\lambda e_{k,j})(0)}.
  \]
 Since
 \begin{align*}
  T_z(T_\lambda e_{k,j})(0)&=\frac 1{(1+|z|^2)^{k/2}} (T_\lambda e_{k,j})(z)
  =\frac{(1+\bar\lambda z)^k}{(1+|z|^2)^{k/2}(1+|\lambda|^2)^{k/2}}\, e_{k,j}(\varphi_\lambda(z))\\
  &= \left[\frac{(1+\bar\lambda z)^2}{(1+|z|^2)(1+|\lambda|^2)}\right]^{k/2} \binom{k}{j}^{1/2} \bigl(\varphi_\lambda(z)\bigr)^j  ,
 \end{align*}
 the sampling inequality in the statement on $p$ and \eqref{eq:one-plus} yield, in particular,
 \begin{align*}
  C&\geq \sum_{\lambda\in\Lambda_k}\sum_{j=0}^{m_\lambda-1} \binom{k}{j} \frac{|1+\bar\lambda z|^{2k}}{(1+|z|^2)^k(1+|\lambda|^2)^k}\, \bigl|\varphi_\lambda(z)\bigr|^{2j}\\
&=\sum_{\lambda\in\Lambda_k} \frac 1{(1+|\varphi_\lambda(z)|^2)^k} \sum_{j=0}^{m_\lambda-1} \binom{k}{j} \bigl|\varphi_\lambda(z)\bigr|^{2j}\\
&\geq \sum_{\substack{
\lambda\in\Lambda_k  \\
z\in D_{\lambda,m_\lambda}}} \frac 1{(1+|\varphi_\lambda(z)|^2)^k} \sum_{j=0}^{m_\lambda-1} \binom{k}{j} \bigl|\varphi_\lambda(z)\bigr|^{2j}
 \end{align*}
From Lemma~\ref{lemma:B} with $x=|\varphi_\lambda(z)|^{2}$ (see Annex) we deduce finally that there exist $c>0$ and $k_0,m_0\in\N$ such that for $k\geq k_0$ and $m_\lambda\geq m_0$
\[
 C\geq \sum_{\substack{
\lambda\in\Lambda_k  \\
z\in D_{\lambda,m_\lambda}}}  c= c \sum_{\lambda\in\Lambda_k} 
\chi_{D_{{\lambda,m_\lambda}}}(z),
\]
so $X$ has finite overlap.

Assume now that $X$ has finite overlap. Applying the local $L^2$-estimates of Lemma~\ref{lemma:local-L2} with $a=0$ to $p\in \mathcal P_k$,
we obtain, 
\begin{align*}
 \sum_{\lambda\in\Lambda_k}\sum_{j=0}^{m_\lambda-1} \bigl|\langle p, T_\lambda e_{k,j}\rangle_k\bigr|^2&\leq
 \sum_{\lambda\in\Lambda_k} C (k+1) \int_{D_{\lambda,m_\lambda}} \frac{|p(z)|^2}{(1+|z|^2)^k}\, d\nu(z)\\
 &\leq C\Bigl[\sup_{k\geq 1}\sup_{z\in\C} \sum_{\lambda\in\Lambda_k} \chi_{D_{\lambda,m_\lambda}}(z)\Bigr]\, \|p\|_k^2.
\end{align*}
\end{proof}

We are ready to prove the necessary conditions for sampling.

 \begin{proof}[Proof of Theorem~\ref{th:sampling-L2}]
  (a) is consequence of Proposition~\ref{prop:overlap_sampling}.
  
  (b) If not, for all $c>0$ and there exists $k(c)\in\N$ such that
  \[
   \C\setminus \bigcup_{\lambda\in\Lambda_{k(c)}} D\bigl(\lambda, \frac{\sqrt{m_\lambda}+c}{\sqrt{k(c)}}\bigr)\neq\emptyset.
  \]
  Take a sequence $c_n$ increasing to $ +\infty$ and $k_n\in\N$ such that
  \[
   w_n\in\C\setminus\bigcup_{\lambda\in\Lambda_{k_n}} D\bigl(\lambda, \frac{\sqrt{m_\lambda}+c_n}{\sqrt {k_n}}\bigr).
  \]
  Then, $w_n$ gets further away from the critical disks; explicitly
  \[
   \rho_n:= d\bigl(w_n,\bigcup_{\lambda\in\Lambda_{k_n}} D\bigl(\lambda, \sqrt{\frac{m_\lambda}{k_n}}\bigr)\bigr)\geq \frac{c_n}{\sqrt{k_n}},
  \]
and the local estimate of Lemma~\ref{lemma:local-L2} applied to $\Lambda_{k(n)}$ and $p=T_w 1$ gives a contradiction with the sampling inequality: on the one hand $\|T_{w_n} 1\|_k=1$, and on the other, by \eqref{eq:one-plus}, the invariance and the finite overlap condition,
\begin{align*}
 \sum_{\lambda\in\Lambda_{k_n}}\sum_{j=0}^{m_\lambda-1} \bigl|\langle T_{w_n} 1, T_\lambda e_{k,j}\rangle_{k_n}\bigr|^2&\lesssim \sum_{\lambda\in\Lambda_{k_n}}\int_{D_{\lambda,m_\lambda}}\frac {|(T_{w_n} 1)(\zeta)|^2}{(1+|\zeta|^2)^{k_n}}\, d\nu(\zeta)\\
 &=\sum_{\lambda\in\Lambda_{k_n}}\int_{D_{\lambda,m_\lambda}}\frac {d\nu(\zeta)}{(1+|\varphi_{w_n}(\zeta)|^2)^{k_n}}\\
 &=\sum_{\lambda\in\Lambda_{k_n}}\int_{D_{\lambda,m_\lambda}} \bigl[1-d^2(w_n,\zeta)]^{k_n} d\nu(\zeta)\\
& \leq C \int_{\C} \bigl[1-d^2(w_n,\zeta)]^{k_n} d\nu(\zeta)\\
 &\lesssim \int_{\C}\Bigl(1-\frac{c_n^2}{k_n}\Bigr)^{k_n} d\nu(\zeta)=\Bigl(1-\frac{c_n^2}{k_n}\Bigr)^{k_n} \stackrel{n\to\infty}{\longrightarrow} 0.
\end{align*}
 \end{proof}
 
Let's prove now the sufficient conditions for sampling.

\begin{proof}[Proof of Theorem~\ref{thm:L2-sampling-suff}]
By hypothesis and by Proposition~\ref{prop:overlap_sampling}, there exist $C>0$ and $k_0\in\N$ such that for all $p\in\P_k$
\[
 \sum_{\lambda\in\Lambda_k}\sum_{j=0}^{m_\lambda-1} |\langle p,T_\lambda e_{k,j}\rangle_k|^2\leq C\|p\|_{k,2}^2.
\]
 Assume that $X$ is not $L^2$-sampling. Then for every $\epsilon>0$ there exist $k(\epsilon)\in\N$ and $p_{k(\epsilon)}\in\P_{k(\epsilon)}$ such that $\|p_{k(\epsilon)}\|_{k(\epsilon),2}=1$ and
 \[
  \sum_{\lambda\in\Lambda_{k(\epsilon)}}\sum_{j=0}^{m_\lambda-1} \bigl|\langle p_{k(\epsilon)}, T_\lambda e_j\rangle_{k(\epsilon)}\bigr|^2\leq\epsilon.
 \]
 By Lemma~\ref{lemma:local-L-inf} there exist $a(\epsilon)>0$ such that $k\geq k_0$ and $m$ big enough
\[
 (k+1)\int_{D\bigl(\lambda,\frac{\sqrt{m_\lambda}-a}{\sqrt k}\bigr)} \frac{|p_{k(\epsilon)}(z)|^2}{(1+|z|^2)^{k(\epsilon)}} d\nu(z)\leq 2\epsilon (k+1)\int_{D_{\lambda,m_\lambda}} \frac{|p_{k(\epsilon)}(z)|^2}{(1+|z|^2)^{k(\epsilon)}} d\nu(z).
\]
Now split
\begin{align*}
 \|p_{k(\epsilon)}\|_{k(\epsilon), 2}^2&=(k+1)\int_{\mathcal U_k(-a)} \frac{|p_{k(\epsilon)}(z)|^2}{(1+|z|^2)^{k(\epsilon)}} d\nu(z)+ (k+1)\int_{\C\setminus\mathcal U_k(-a)} \frac{|p_{k(\epsilon)}(z)|^2}{(1+|z|^2)^{k(\epsilon)}} d\nu(z).
\end{align*}
With the previous estimate on the first term and Lemma~\ref{lemma:pointwise} on the second, together with the finite overlap condition,
\begin{align*}
 1=&\|p_{k(\epsilon)}\|_{k(\epsilon), 2}^2\leq\sum_{\lambda\in\Lambda_{k(\epsilon)}}
 (k+1)\int_{D\bigl(\lambda,\frac{\sqrt{m_\lambda}-a}{\sqrt k}\bigr)} \frac{|p_{k(\epsilon)}(z)|^2}{(1+|z|^2)^{k(\epsilon)}} d\nu(z)+ (k+1)\int_{\C\setminus\mathcal U_k(-a)}  d\nu(z)\\
 \leq & 2\epsilon \sum_{\lambda\in\Lambda_{k(\epsilon)}}
 (k+1)\int_{D_{\lambda,m_\lambda}} \frac{|p_{k(\epsilon)}(z)|^2}{(1+|z|^2)^{k(\epsilon)}} d\nu(z)
 + (k+1) \nu\bigl(\C\setminus\mathcal U_k(-a)\bigr)\\
 \leq & 2\epsilon S_X \|p_{k(\epsilon)}\|_{k(\epsilon), 2}^2+ (k+1) \nu\bigl(\C\setminus\mathcal U_k(-a)\bigr)= 2\epsilon S_X+ (k+1) \nu\bigl(\C\setminus\mathcal U_k(-a)\bigr).
\end{align*}
Take now $\epsilon<1/(2S_X)$ and fix $a=a(\epsilon)$ as above (observe that $a$ depends only on $X$). Then, taking the inequality above to the limit as $k\to\infty$ we reach a contradiction.
\end{proof}

 \section{Proof of Theorem~\ref{thm:intL2}}\label{sec:int}
 
(a) Assume that the conclusion does not hold. Then for all $c>1$ there are infinitely many $k$ for which there exist $\lambda,\lambda'\in\Lambda_k$ with
\[
D\Bigl(\lambda, \frac{\sqrt{m_\lambda}-(c+1)}{\sqrt k}\Bigr)\cap D\Bigl(\lambda', \frac{\sqrt{m_\lambda'}-(c+1)}{\sqrt k}\Bigr)\neq\emptyset.
\]
In particular, there exist $w_k\in\C$ such that
\[
 D\bigl(w_k,1/\sqrt k\bigr)\subset D\Bigl(\lambda, \frac{\sqrt{m_\lambda}-c}{\sqrt k}\Bigr)\cap D\Bigl(\lambda', \frac{\sqrt{m_\lambda'}-c}{\sqrt k}\Bigr).
\]
Since $X$ is $L^2$-interpolating, there exist $p_k\in\P_k$ vanishing at $\lambda$ at order $m_\lambda$ and coinciding with
$T_{w_k}1$ at order $m_{\lambda'}$ on $\lambda'$; more precisely, such that
\begin{itemize}
 \item $p\in N_{\lambda,m_\lambda}$, that is $\langle p, T_\lambda e_{k,j}\rangle_k=0$ for all $j<m_{\lambda}$,
 \item $p-T_{w_k} 1\in N_{\lambda',m_{\lambda'}}$,
 \item $\|p\|_{k,2}\leq M_X$ ($M_X$ interpolating constant).
\end{itemize}
Then
\begin{align*}
 \int_{D(w_k,\frac 1{\sqrt{k}})}\frac {|(T_{w_k}1)(z)|^2}{(1+|z|^2)^k}\, d\nu(z)&\leq \int_{D(w_k,\frac 1{\sqrt{k}})}\frac {|(T_{w_k}1-p)(z)|^2}{(1+|z|^2)^k}\, d\nu(z)+
 \int_{D(w_k,\frac 1{\sqrt{k}})}\frac {|p(z)|^2}{(1+|z|^2)^k}\, d\nu(z)\\
 \leq\int_{D\bigl(\lambda',\frac{\sqrt{m_\lambda'}-c}{\sqrt k}\bigr)} &\frac {|(T_{w_k}1-p)(z)|^2}{(1+|z|^2)^k}\, d\nu(z)+
 \int_{D\bigl(\lambda,\frac{\sqrt{m_\lambda}-c}{\sqrt k}\bigr)}\frac {|p(z)|^2}{(1+|z|^2)^k}\, d\nu(z).
\end{align*}
By the local estimates of Lemma~\ref{lemma:local-L-inf} (on both $p$ at $\lambda$ and $p-T_{w_k} 1$ at $\lambda'$) this implies
\[
 I_k:=(k+1)\int_{D(w_k,\frac 1{\sqrt{k}})}\frac {|(T_{w_k}1)(z)|^2}{(1+|z|^2)^k}\, d\nu(z)=\varepsilon(c) M_X^2,
\]
where $\varepsilon(c)>0$ is such that $\lim\limits_{c\to\infty} \varepsilon(c)=0$.

But, on the other hand, it is clear that the left-hand side of this estimate is bounded below: by the invariance, and computing as in the proof of Lemma~\ref{lemma:pointwise},
\begin{align*}
  I_k&=(k+1)\int_{D(0,\frac 1{\sqrt{k}})}\frac {d\nu(z)}{(1+|z|^2)^k}
  =(k+1)\int_{\Delta\bigl(0,\sqrt{\frac {1/k}{1-1/k}}\bigr)}\frac {d\nu(z)}{(1+|z|^2)^k}\\
  &=1-\bigl(1-\frac 1k\bigr)^{k+1}\stackrel{k\to\infty}{\xrightarrow{\hspace{1.5cm}}} 1-e^{-1}>0.
\end{align*}
We reach therefore a contradiction.

\bigskip 

(b) We adapt to this setting the techniques introduced in \cite{BOC}.

Let $v_X^k=\bigl\{v_\lambda\bigr\}_{\lambda\in\Lambda_k}$, $k\geq 1$, with $v_\lambda=(v_\lambda^j)_{ j<m_\lambda}$, such that
\[
 \sup_{k\geq 1} \|v_X^k\|_{\ell^2}^2=\sup_{k\geq 1}\Bigl(\sum_{\lambda\in\Lambda_k}\sum_{j<m_\lambda} |v_\lambda^j|^2\Bigr)<+\infty.
\]
Take polynomials $(p_\lambda)_{\lambda\in\Lambda_k}$ such that $\langle p, e_{k,j}\rangle_k=v_\lambda^j$ for
$\lambda\in\Lambda_k$, $j<m_\lambda$, and
\[
\bigl\|p_\lambda\bigr\|_{\P_k/N_{0,m_\lambda}}=\|p_\lambda\|_{k,2}\ ,\qquad \sup_{k\geq 1} 
\sum_{\lambda\in \Lambda_k} \bigl\|p_\lambda\bigr\|_{\P_k/N_{0,m_\lambda}}^2<+\infty,
\]
which exist by hypothesis. 
Then the polynomials $q_\lambda=T_\lambda p_\lambda$ provide local interpolation
\[
 \langle q_\lambda, T_\lambda e_{k,j}\rangle_k=\langle T_\lambda p_\lambda, T_\lambda e_{k,j}\rangle_k=\langle p_\lambda, e_{k,j}\rangle_k=v_\lambda^j\qquad \lambda\in\Lambda_k,\ j<m_\lambda
\]
with controlled norms:
\begin{equation}\label{estimate-ql}
 \sum_{\lambda\in \Lambda_k} \bigl\|q_\lambda\bigr\|_{k,2}^2=\sum_{\lambda\in \Lambda_k} \bigl\|p_\lambda\bigr\|_{k,2}^2= \sum_{\lambda\in \Lambda_k} \bigl\|p_\lambda\bigr\|_{\P_k/N_{0,m_\lambda}}^2\lesssim 1.
\end{equation}

Let
\[
 D_{\lambda, m_\lambda}=D\bigl(\lambda,\frac{\sqrt{m_\lambda}}{\sqrt k}\bigr)\qquad D_{\lambda, m_\lambda}^\prime
 =D\bigl(\lambda,\frac{\sqrt{m_\lambda}+c}{\sqrt k}\bigr).
\]
By hypothesis, for $k$ big, the disks $\{D_{\lambda,m_\lambda}^\prime\}_{\lambda\in\Lambda_k}$, are pairwise disjoint.

Now, using this separation, we shall glue together the previous local interpolating functions and produce a smooth interpolating function with the desired growth. To that purpose consider $\eta\in\mathcal C^\infty(\R)$  such that
 \[
  \eta(x)=
  \begin{cases}
   1\quad&\text{for $x<-c$}\\
   0\quad&\text{for $x>0$},
  \end{cases}
 \]
 with $|\eta'|\lesssim1$, and define, for $k\geq 1$,
 \[
  F_k(z)=\sum_{\lilk}q_\lambda(z)\, \eta\bigl(\sqrt k\ d(\lambda,z)-\sqrt{m_\lambda}-c\bigr).
 \]
 Observe that  $\supp(F_k)\subseteq \cup_{\lambda\in\Lambda_k}D_{\lambda,m_\lambda}^\prime$ and that $F_k$ has 
the characteristic growth of $L^2$ polynomials in $\P_k$: for $z\in D_{\lambda,m_\lambda}^\prime$
\[
 |F_k(z)|=|q_\lambda(z)| \eta\bigl(\sqrt k\ d(\lambda,z)-\sqrt{m_\lambda}-c\bigr) \leq |q_\lambda(z)|.
\]
Hence, by \eqref{estimate-ql}, 
\begin{equation}\label{estimate-Fk}
 (k+1)\int_{\C}\frac{|F_k(z)|^2}{(1+|z|^2)^k}\, d\nu(z)\leq  \sum_{\lambda\in\Lambda_k}(k+1)\int\limits_{D_{\lambda,m_\lambda}^\prime} \frac{|q_\lambda(z)|^2}{(1+|z|^2)^k}\, d\nu(z)\leq \sum_{\lambda\in\Lambda_k}\|q_\lambda\|_2^2\lesssim 1.
\end{equation}

We want to correct $F_k$ to produce a holomorphic polynomial with the right growth and interpolating the prescribed values on each $\lambda\in\Lambda_k$. We look for a function of the form $p_k=F_k-u_k$, so we want to find solutions to the equation $\bar\partial u_k=\bar\partial F_k$ with (weighted) $L^2$-estimates.

Observe that $\bar\partial F_k$ is supported in the union of the chordal annuli $D_{\lambda,m_\lambda}^\prime\setminus D_{\lambda,m_\lambda}$. For $z$ in such annulus,
\[
\frac{\partial F_k}{\partial\bar z}(z)=q_\lambda(z)\, \eta'\bigl(\sqrt k\ d(\lambda,z)-\sqrt{m_\lambda}-c\bigr) \sqrt k\, \frac{\partial}{\partial \bar z}\, d(\lambda,z),
\]
hence
\[
 |\bar\partial F_k(z)|^2\lesssim k\, |q_\lambda(z)|^2  \bigl|\frac{\partial}{\partial \bar z}\, d(\lambda,z)\bigr|.
\]
Since
\[
 d(\lambda,z)=\frac{|z-\lambda|}{(1+|z|^2)^{1/2} (1+|\lambda|^2)^{1/2}}
\]
we have
\begin{align*}
 \frac{\partial}{\partial \bar z}\, d(\lambda,z)&=\frac 1{(1+|\lambda|^2)^{1/2}}\left[\frac{\frac 12\frac{z-\lambda}{|z-\lambda|}(1+|z|^2)^{1/2}-|z-\lambda|\frac 12 (1+|z|^2)^{-1/2}\, z}{1+|z|^2}\right]\\
 &=\frac{1/2}{(1+|\lambda|^2)^{1/2} (1+|z|^2)^{3/2}} \frac{z-\lambda}{|z-\lambda|}\bigl[1+|z|^2-(\bar z-\bar\lambda)z\bigr]\\
&=\frac 12\frac{1+\lambda z}{(1+|\lambda|^2)^{1/2} (1+|z|^2)^{3/2}} \frac{z-\lambda}{|z-\lambda|}.
\end{align*}
With this and \eqref{eq:one-plus},
\begin{align*}
 (1+|z|^2)^2 |\bar\partial F_k(z)|^2 &\lesssim k\, |q_\lambda(z)|^2 \frac{|1+\lambda z|^2}{(1+|\lambda|^2) (1+|z|^2)}= \frac{k\, |q_\lambda(z)|^2}{1+|\varphi_\lambda(z)|^2}.
\end{align*}

Since $1+|\varphi_\lambda(z)|^2\leq 1+r_{\lambda,m_\lambda}^2(c)\lesssim 1$ for $z\in D_{\lambda,m_\lambda}^ \prime\setminus D_{\lambda,m_\lambda}$, we deduce that
\begin{equation}\label{estimate-db}
 (1+|z|^2)^2\Bigl|\frac{\partial F_k}{\partial \bar z}(z)\Bigr|^2\lesssim k\, |q_\lambda(z)|^2\ \qquad z\in D_{\lambda,m_\lambda}^ \prime\setminus D_{\lambda,m_\lambda}.
\end{equation}

Now we solve the $\bar\partial$-equation using H\"ormander's estimates  (see \cite[(4.2.6)]{Hor}): given $\phi$ subharmonic in $\C$ there exists a solution $u_k$ to the equation $\bar\partial u_k=\bar\partial F_k$ such that
\[
 \int_{\C} |u_k|^2\, e^{-\phi} \leq \int_{\C} |\bar\partial F_k|^2\, \frac{e^{-\phi}}{\Delta\phi},
\]
where here $\Delta\phi=\dfrac{\partial^2\phi}{\partial z\partial\bar z}$.

Take 
\[
 \phi_k(z)=(k+2)\log (1+|z|^2)+v_k(z),
\]
where
\[
 v_k(z)=\sum_{\lambda\in\Lambda_k} m_\lambda\, \left[\log\bigl(\frac k{m_\lambda} d^2(z,\lambda)\bigr)+1-(\frac k{m_\lambda} d^2(z,\lambda)\bigr)\right]\, \chi_{D_{\lambda,m_\lambda}}(z).
\]
By the separation hypothesis $v_k\in \mathcal C^1(\C)$. Also, $\supp(v_k)\subset\cup_{\lambda\in\Lambda_k} D_{\lambda,m_\lambda}$. In particular,
\[
 \supp(v_k)\cap \supp (\bar \partial F_k)=\emptyset.
\]
Then, since $v_k\leq 0$ by construction,
\begin{align*}
 \int_{\C}\frac{|u_k(z)|^2}{(1+|z|^2)^k}\, d\nu(z)&=\frac 1{\pi}\int_{\C}\frac{|u_k(z)|^2}{(1+|z|^2)^{k+2}}\, dm(z)\leq \frac 1{\pi}\int_{\C} |u_k|^2\, e^{-\phi_k} 
 \leq \frac 1{\pi}\int_{\C} |\bar\partial F_k|^2\, \frac{e^{-\phi_k}}{\Delta\phi_k}.
\end{align*}
As mentioned above, on $\supp(\bar \partial F_k)$ one has $\phi_k(z)=(k+2)\log (1+|z|^2)$, so that $e^{-\phi_k(z)}=(1+|z|^2)^{-(k+2)}$ and
\[
 \Delta\phi_k(z)= \frac{k+2}{(1+|z|^2)^2}.
\]
Hence, by \eqref{estimate-db}, 
\begin{align*}
 \int_{\C}\frac{|u_k(z)|^2}{(1+|z|^2)^k}\, d\nu(z)&\leq 
 \frac 1{\pi}\int_{\C} |\bar\partial F_k(z)|^2\, \frac{(1+|z|^2)^{-(k+2)}}{(k+2)(1+|z|^2)^{-2}}\, dm(z)\\
& =\frac 1{k+2} \int_{\C} (1+|z|^2)^2 |\bar\partial F_k(z)|^2 \, \frac{d\nu(z)}{(1+|z|^2)^k}\\
&\lesssim\frac {k}{k+2} \sum_{\lilk}\int_{D_{\lambda,m_\lambda}^\prime\setminus D_{\lambda,m_\lambda}} |q_\lambda(z)|^2 \, \frac{d\nu(z)}{(1+|z|^2)^k}
\lesssim \sum_{\lilk}\int_{\C} \frac{|q_\lambda(z)|^2}{(1+|z|^2)^k} \, d\nu(z).
\end{align*}
From this and \eqref{estimate-ql} we deduce that
\begin{align*}
 \|u_k\|_{k,2}^2=(k+1)\int_{\C}\frac{|u_k(z)|^2}{(1+|z|^2)^k}\, d\nu(z)&\lesssim\sum_{\lilk}\|q_\lambda\|_{k,2}^2\lesssim 1, 
\end{align*}
which together with \eqref{estimate-Fk} shows that $p_k:=F_k-u_k$ has $\|p_k\|_{k,2}\leq \|F_k\|_{k,2}+\|u_k\|_{k,2} $ finite.

It remains to check that $p_k$ interpolates with the prescribed multiplicity on each $\lilk$. Notice that  $e^{-v_k}$ has a singularity of order $2m_\lambda$ on each $\lilk$: in a small neighborhood of such $\lambda$
\[
 e^{-v_k(z)}\sim e^{-m_\lambda\log(\frac k{m_\lambda} d^2(z,\lambda)\bigr)}=\frac{(m_\lambda/k)^{m_\lambda}}{d(z,\lambda)^{2m_\lambda}}\sim \frac 1{|z-\lambda|^{2m_\lambda}}.
\]
Then the integrability of $|u_k|^2 e^{-\phi_k}$ seen above forces $u_k$ to vanish at order $m_\lambda$ on $\lambda$: for $\epsilon_\lambda$ small one has
\begin{align*}
 \int_{D(\lambda,\epsilon_\lambda)} |u(z)|^2 e^{-\phi_k(z)}=\int_{D(\lambda,\epsilon_\lambda)}\frac{|u(z)|^2}{(1+|z|^2)^k}\, e^{-v_k(z)}  <+\infty,
\end{align*}
hence
\[
\int_{D(\lambda,\epsilon_\lambda)}\frac{|u(z)|^2}{(1+|z|^2)^k}\,\frac {dm(z)}{|z-\lambda|^{2m_\lambda}}<+\infty
\]
and $u_k$ must vanish at least at order $m_\lambda$, that is, for $j<m_\lambda$,
\[
 \langle p_k, T_\lambda e_{k,j}\rangle_k=\langle F_k, T_\lambda e_{k,j}\rangle_k=\langle q_\lambda, T_\lambda e_{k,j}\rangle_k=v_\lambda^j.
\]

\section{Annex}

This section is dedicated to gloomy computations on the incomplete beta and binomial functions. In particular, we prove Lemma~\ref{lemma:incomplete-beta}.

We begin with an estimate that is instrumental in the proof of the bounds for the
incomplete function $I(m,k+1-m; x)$.

\begin{lemma}\label{lemma:0}
 Let $k\in\N$ and let $m_k\in\N$ be such that $\lim_k m_k/k=0$. Then there exist $k_0,m_0\in\N$ such that for $k\geq k_0$ and $m=m_0,\dots, m_k$
 \[
 \frac {k^m \Gamma(k-m+1)}{\Gamma(k+1)}=\frac{k^{m}(k-m)!}{k!}\leq e^m\Bigl(1-\frac mk\Bigr)^{k-m}.
 \]
\end{lemma}

\begin{remark*}
The identity 
\[
-\log\bigl(1-\frac mk\bigr)=\log \bigl(1+\frac m{k-m}\bigr)
\]
together with the development
\[
 \log(1+t)=\sum_{n=1}^\infty\frac{(-1)^{n+1}}n\ t^n\ \qquad |t|<1
\]
show that
\[
 \frac{k^{m} (k-m)!}{k!}\leq \exp\left(\sum_{n=2}^\infty \frac{(-1)^n}n\frac {m^n}{(k-m)^n}\right).
\]
Writing
 \[
  \frac 1{k-m}=\frac 1k+\frac m{k(k-m)}
 \]
and taking the leading terms we see that
\[
 \frac{k^{m} (k-m)!}{k!}\leq \exp\left(\frac 12\frac{m^2}k+\frac 16\frac{m^3}{k^2}++\frac 1{12}\frac{m^4}{k^3}+\cdots\right)
 = e^{\frac 12\frac{m^2}k+o(\frac{m^2}k)}.
\]
\end{remark*}

\begin{proof} Denote
\[
 A_{k,m}:=\frac{k^{m} (k-m)!}{k!}=\frac k{k-m+1}\, \frac k{k-m+2}\cdots \frac kk=\prod_{j=1}^{m-1}\frac k{k-j},
\]
so that
\[
 \log A_{k,m}=\sum_{j=1}^{m-1} \log\bigl(\frac k{k-j}\bigr) .
\]
Since the function $f(x)=\log\bigl(\frac k{k-x}\bigr)$ is increasing in $x$, we can estimate this sum by the integral:
\begin{align*}
  \log A_{k,m}&\leq \int_0^m \log\bigl(\frac k{k-x}\bigr)\, dx=m\log k -\int_0^m \log(k-x)\, dx\\
  &=m\log k +\bigl [(k-x)\log(k-x)-(k-x)\bigr]_0^m\\
  &=m\log k+(k-m)\log(k-m)-(k-m)-k\log k +k\\
  &=m\log\bigl(\frac k{k-m}\bigr)-k\log\bigl(\frac k{k-m}\bigr)+m\\
  &=m-(k-m)\log\bigl(\frac k{k-m}\bigr)=m+(k-m)\log\bigl(1-\frac mk\bigr).
\end{align*}
\end{proof}

The following results are based on the elementary observation that 
most of the mass of the density function $\psi_{m,k}(t)=t^m (1-t)^{k-m}$ is accumulated near its peak, at $m/k$. This is easy to check, since $\psi_{m,k}$ is increasing from 0 to $m/k$, decreasing from $m/k$ to 1, and is convex in the interval
$\frac mk\pm\frac{\sqrt m}k\sqrt{1-\frac{m-1}{k-1}}$.
This observation will be relevant in the proof of Lemma~\ref{lemma:incomplete-beta}.

\begin{lemma}\label{lemma:tip-beta}
  Let $k\in\N$ and let $m_k\in\N$ be such that $\lim_k m_k/k=0$. Then there exist $c>0$,  $k_0,m_0\in\N$ such that for $k\geq k_0$ and $m=m_0,\dots, m_k$
 \[
  \int_{\frac{m-\sqrt m}k}^{\frac mk}t^m (1-t)^{k-m} dt\geq c\ \frac{\Gamma(m+1)\Gamma(k-m+1)}{\Gamma(k+2)}=c\, \beta(m+1,k-m+1).
 \]
\end{lemma}

\begin{proof}
 Since $\frac mk\leq\frac{m_k}k\to 0$, the function $\psi_{m,k}(t)=t^m (1-t)^{k-m}$ is convex in the integration domain, and we can estimate below using the trapezoidal rule:
 \begin{align*}
   \int_{\frac{m-\sqrt m}k}^{\frac mk}t^m (1-t)^{k-m} dt&\geq \frac{\sqrt m}k\, \frac 12\, \bigl[\psi_{m,k}(\frac mk)+\psi_{m,k}(\frac {m-\sqrt m}k)\bigr]\geq \frac 12\, \frac{\sqrt m}k\, \psi_{m,k}(\frac mk)\\
   &=\frac 12 \, \frac{\sqrt m}k\, (\frac mk)^m\bigl(1-\frac mk\bigr)^{k-m}=\frac 12\, \frac{\sqrt m\, m^m e^{-m}}{k^{m+1}}\, e^m \bigl(1-\frac mk\bigr)^{k-m}.
 \end{align*}
 By Stirling's formula $\Gamma(m+1)=m!\simeq \sqrt{2\pi m}\, m^m e^{-m}$, we shall have the desired estimate as soon as there exist $c>0$, $k_0,m_0\in\N$ such that for $k\geq k_0$  and $m=m_0,\dots,m_k$
 \[
  e^m \bigl(1-\frac mk\bigr)^{k-m}\geq c\ \frac{k^{m+1}\Gamma(k-m+1)}{\Gamma(k+2)}= c\ \frac{k}{k+1}\, \frac{k^{m}\Gamma(k-m+1)}{\Gamma(k+1)}.
 \]
This holds, according to Lemma~\ref{lemma:0}.
\end{proof}

In the same spirit of the previous lemma we have the following.

\begin{lemma}\label{lemma:estimates-a}
 (a) There exist $k_0,m_0\in\N$ such that for $k\geq k_0$, $m\geq m_0$ and for all $a>0$
\[
 (s-a\sqrt s)^m\bigl(1-\frac {s-a\sqrt s}k\bigr)^{k-m}\geq e^{-3a^2} s^m \bigl(1-\frac {s}k\bigr)^{k-m}\quad
 \textrm{  $s\in (m-\sqrt m,m)$}.
\]
 
 (b) Let $a>0$ and let $m_k\in\N$ be such that $\lim_k m_k/k=0$. Then, for $j=m,\dots,k$ and for $s\in[a^2,m]$
 \[
  (s-a\sqrt s)^j\bigl(1-\frac{s-a\sqrt s}{k}\bigr)^{k-j}\leq e^{-\frac 12 a^2} s^j \bigl(1-\frac sk\bigr)^{k-j}.
 \]
\end{lemma}

\begin{proof}
 (a) There is no restriction in assuming that $a>1$. The inequality to be proved is
\[
 \bigl( 1-\frac a{\sqrt s}\bigr)^m \bigl(1+\frac {a\sqrt s}{k-s}\bigr)^{k-m}\geq e^{-3a^2},
\]
which is
\[
 m\log \bigl (1-\frac a{\sqrt s}\bigr)+(k-m)\log\bigl(1+\frac {a\sqrt s}{k-s}\bigr) \geq -3a^2.
\]
Observe that, by hypothesis, and since $s\leq m$, both $a/\sqrt s$ and $a\sqrt s/(k-s)$ tend to $0$ as $k$ increases. Take $t_0>0$ small enough so that for $t\in(0,t_0)$
\begin{align*}
 &\log(1-t)=-\sum_{n=1}^\infty\frac {t^n}n\geq -t-t^2 \\
 &\log(1+t)=\sum_{n=1}^\infty\frac{(-1)^{n+1}}n\, t^n\geq t-\frac{t^2}2.
\end{align*}
Then there exists $m_0\in\N$ such that for $m\geq m_0$,
\begin{align*}
 m\log \bigl (1-\frac a{\sqrt s}\bigr)+(k-m)\log\bigl(1+\frac {a\sqrt s}{k-s}\bigr)
 &\geq -a\frac m{\sqrt s}-a^2\frac {m}{s}+(k-m)\bigl(\frac {a\sqrt s}{k-s}-\frac 12\frac{a^2 s}{(k-s)^2}\bigr) \\
\geq  &-a\frac m{\sqrt s}-a^2\frac {m}{s}+\frac{k-m}{k-s}a\sqrt s-\frac {a^2}2\, \frac {(k-m)s}{(k-s)^2}.
\end{align*}
Thus, it is enough to see that for $k\geq k_0$ and $m\geq m_0$ the functions
\[
 H_{m,k}(s)=a\frac m{\sqrt s}+a^2\frac {m}{s}-\frac{k-m}{k-s}a\sqrt s+\frac {a^2}2\, \frac {(k-m)s}{(k-s)^2}
\]
are bounded above in $(m-a\sqrt m, m)$ by $3a^2$, for some $c>0$. 

The second and last term can be disregarded at no cost, since for $s\in (m-a\sqrt m,m)$ and $m\geq 4a^2$
\[
\frac ms\leq \frac{m}{m-a\sqrt m}=\frac 1{1-a/\sqrt m}\leq 2
\] 
and 
\[
 \frac {(k-m)s}{(k-s)^2}\leq \frac {(k-m)m}{(k-m)^2}=\frac{m/k}{1-m/k}\stackrel{k\to\infty}{\longrightarrow}\, 0.
\]
Rename thus $H_{m,k}$ as
\[
 H_{m,k}(s)=a\frac m{\sqrt s}-\frac{k-m}{k-s}a\sqrt s=\frac a{1-s/k}\frac{m-s}{\sqrt s}.
\]
Now the estimate is clear: for $k\geq k_0$
\[
 \frac a{1-s/k}\leq \frac a{1-m/k}\leq  2a,
\]
so that
\[
 H_{m,k}(s)\leq 2a\frac{m-s}{\sqrt s}\leq 2a\, \frac{m-(m-a\sqrt m)}{\sqrt {m-a\sqrt m}}=\frac{2a^2}{\sqrt{1- \frac a{\sqrt m}}} .
\]

\bigskip

(b) The estimates to be proved are now, for $s\in[a^2,m]$,
\[
 \bigl(1-\frac a{\sqrt s}\bigr)^j \bigl(1+\frac{a\sqrt s}{k-s}\bigr)^{k-j}\leq e^{-\frac 12 a^2},
\]
which are
\[
 j\log\bigl(1-\frac a{\sqrt s}\bigr)+(k-j)\log \bigl(1+\frac{a\sqrt s}{k-s}\bigr)\leq -\frac 12 a^2.
\]
Observe that both $a/\sqrt s$ and $\sqrt s/(k-s)$ are positive and close to $0$. Using that, for $t\in (0,1)$,
\[
 \log(1-t)\leq -t-\frac{t^2}2\quad ,\quad \log(1+t)\leq t
\]
we see that
\begin{align*}
 j\log\bigl(1-\frac a{\sqrt s}\bigr)+(k-j)\log \bigl(1+\frac{a\sqrt s}{k-s}\bigr)&\leq-j\bigl(\frac a{\sqrt s}+\frac{a^2}{2s}\bigr)+(k-j)\frac{a\sqrt s}{k-s}\\
 =-\frac{a^2}2 \frac js-a\frac j{\sqrt s}+(k-j)\frac {a\sqrt s}{k-s}.
\end{align*}
Since $s\leq m\leq j$ (and $j/s\geq 1$), it will be enough to see that
\[
 G_{m,k}(s):= -a\frac j{\sqrt s}+(k-j)\frac {a\sqrt s}{k-s}\leq 0.
\]
This is immediate, since $j\geq m$, $G_{m,k}$ is increasing (both terms are) and
\[
 G_{m,k}(m)=a\sqrt m\bigl[\frac{k-j}{k-m}-\frac jm\bigr]=a\sqrt m\, \frac {k(m-j)}{(k-m)\sqrt m}=-\frac{a(j-m)}{k-m}.
\]
\end{proof}

We are finally ready to prove Lemma~\ref{lemma:incomplete-beta}.

\emph{Proof of (a)}. Since for $x\in (0,1)$
   \[
    I(j+1,k-j+1; x)=\sum_{l=j+1}^{k+1} \binom{k+1}{l} x^l (1-x)^{k+1-l}
   \]
it is clear that $I(j+1,k-j+1; x)$ is decreasing in $j$. Therefore it is enough to prove that
for $a>0$ there exists $\epsilon(a)>0$ and $k_0, m_0\in\N$ such that for all $k\geq k_0$ and  $m\in[ m_0,m_k]$
\[
I\bigl(m,k-m+2; \frac{m-a\sqrt m}{k}\bigr)\geq \epsilon(a),
\]
or equivalently
 \[
  \int_0^{\frac{m-a\sqrt m}k} t^{m-1} (1-t)^{k-m+1} dt\geq \epsilon(a) \frac{\Gamma(m)\Gamma(k-m+2)}{\Gamma(k+2)}=\epsilon(a) \, \beta(m,k-m+2).
 \]
 With the substitution $t=s/k$ we have
\begin{align}\label{int:subst}
 \int_0^{\frac{m-a\sqrt m}k} t^{m-1} (1-t)^{k-m+1} dt&=\frac 1{k^m}\int_0^{m-a\sqrt m} s^{m-1} \bigl(1-\frac sk\bigr)^{k-m+1} ds\\
 &\geq \frac 1{k^m}\int_{m-2a\sqrt m}^{m-a\sqrt m} s^{m-1} \bigl(1-\frac sk\bigr)^{k-m+1} ds\nonumber.
\end{align}

Substitute now $s-a\sqrt s=u$,
\begin{align*}
 e^{-3a^2}\int_{m-a\sqrt m}^m s^{m-1} \bigl(1-\frac {s}k\bigr)^{k-m+1} ds&\leq \int_{m-a\sqrt m}^m(s-a\sqrt s)^{m-1}\bigl(1-\frac {s-a\sqrt s}k\bigr)^{k-m+1} ds\\
 &= \int\limits_{m-a\sqrt m-a\sqrt{m-a\sqrt m}}^{m-a\sqrt m} u^{m-1} \bigl(1-\frac {u}k\bigr)^{k-m+1}\frac{du}{1-\frac a{2\sqrt s}}.
\end{align*}
The last integration interval is contained in $(m-2a\sqrt m,m-a\sqrt m)$, where, for $m\geq m_0$,
\[
 1-\frac a{2\sqrt s}\geq 1-\frac a{2\sqrt {m-a\sqrt m}}\geq \frac 12.
\]
Then 
 \[
 e^{-3a^2}\int_{m-a\sqrt m}^m s^{m-1} \bigl(1-\frac {s}k\bigr)^{k-m+1} ds
  \leq 2\int_{m-2a\sqrt m}^{m-a\sqrt m} u^{m-1} \bigl(1-\frac {u}k\bigr)^{k-m+1} du.
 \]
From \eqref{int:subst} and reverting the substitution $t=s/k$,
\begin{align*}
 \int_0^{\frac{m-a\sqrt m}k} t^{m-1} (1-t)^{k-m+1} dt&\geq\frac {e^{-3a^2}}2 \frac 1{k^m}\int_{m-a\sqrt m}^m s^{m-1} \bigl(1-\frac {s}k\bigr)^{k-m+1} ds\\
 &=\frac {e^{-3a^2}}2 \int_{\frac{m-a\sqrt m}k}^{\frac mk}t^{m-1} (1-t)^{k-m+1} dt.
\end{align*}
The result then follows from Lemma~\ref{lemma:tip-beta}.

\bigskip

\emph{Proof of (b)}.
The inequalities to be proved are
 \[
  \int_0^{\frac{m-a\sqrt m}{k}} t^j (1-t)^{k-j} dt\leq \epsilon\, \int_0^{\frac{m-a\sqrt m}{k}} t^j (1-t)^{k-j} dt \qquad j=m,\dots, k.
 \]
 After the substitution $t=s/k$, this is
  \[
  \int_0^{m-a\sqrt m} s^j \bigl(1-\frac sk\bigr)^{k-j} ds\leq \epsilon\, \int_0^{m} s^j \bigl(1-\frac sk\bigr)^{k-j} ds \qquad j=m,\dots, k.
 \]
By Lemma~\ref{lemma:estimates-a}(b),
\begin{align*}
 \int_{a^2}^m (s-a\sqrt s)^j\bigl(1-\frac{s-a\sqrt s}{k}\bigr)^{k-j} ds&\leq 
 e^{-\frac 12 a^2} \int_{a^2}^m s^j \bigl(1-\frac sk\bigr)^{k-j} ds\\
 &\leq 
 e^{-\frac 12 a^2} \int_{0}^m s^j \bigl(1-\frac sk\bigr)^{k-j} ds.
\end{align*}
On the other hand, with the change $s-a\sqrt s=u$, and since $1/2\leq 1-a/(2\sqrt s)<1$ for $s\in[a^2,m]$,
\begin{align*}
  \int_{a^2}^m (s-a\sqrt s)^j\bigl(1-\frac{s-a\sqrt s}{k}\bigr)^{k-j} ds&=\int_0^{m-a\sqrt m} u^j \bigl(1-\frac uk\bigr)^{k-j} \frac{du}{1-\frac a{2\sqrt s}}\\
  &\geq \int_0^{m-a\sqrt m} u^j \bigl(1-\frac uk\bigr)^{k-j} du.\quad\square
\end{align*}

\bigskip

We finish with a related estimate for the incomplete binomial sum.

 \begin{lemma}\label{lemma:B}
 Let $m_k\in\N $ be a sequence such that $\lim_k \frac{m_k}k=0$.
There exist $c\in (0,1)$ and $k_0, m_0\in\N$ such that for all $k\geq k_0$ and $m\in [ m_0,m_k]$
  \[
   F_{k,m}(x):=\frac 1{(1+x)^k} \sum_{j<m} \binom{k}{j}\, x^j\geq c\qquad \textrm{for}\quad 0\leq  x\leq\frac{m/k}{1-m/k}.
  \]
 \end{lemma}
 

\begin{proof}
 Let's see first that $F_{k,m}$ is a decreasing function. Since
 \[
  F_{k,m}^\prime(x)=\frac{\bigl[\sum\limits_{j<m}\binom{k}{j}\, j x^{j-1}\bigr] (1+x) -\bigl[\sum\limits_{j<m}\binom{k}{j} x^j\bigr]\, k}{(1+x)^{k+1}},
 \]
 the identity
 \[
\binom{k}{j+1}(j+1)+ \binom{k}{j} j=k \binom{k}{j}
 \]
yields 
 \begin{align*}
  (1+x)^{k+1} F_k^\prime(x)&=\sum_{j<m}\binom{k}{j}\, j x^{j-1}+\sum_{j<m}\binom{k}{j}\, j x^{j}-\sum_{j<m}
  \binom{k}{j}\, k x^{j}\\
&=-\binom{k}{m-1}(k-m+1) x^{m-1}+\sum_{j<m-1}\bigl[
\binom{k}{j+1} (j+1)+\binom{k}{j} j -\binom{k}{j} k\bigr]\, x^j\\
&=-\binom{k}{m-1}(k-m+1) x^{m-1} \leq 0.
 \end{align*}

 Since $F_{k,m}$ is decreasing it is enough to prove that there exist $c\in(0,1)$ and $k_0\in\N$ such that
 \[
  F_{k,m}(\rho_{k,m})=\frac 1{(1+\rho_{k,m})^k}\sum_{j<m} \binom{k}{j} \rho_{k,m}^j\geq 
  c\qquad\text{where}\quad \rho_{k,m}=\frac{m/k}{1-m/k}. 
 \]
 
 Let us observe next that the factors $\binom{k}{j} \rho_{k,m}^j$
increase in $j=0,\dots,m-1$: the inequality
\[
 \binom{k}{j} \rho_{k,m}^j \leq
\binom{k}{j+1} \rho_{k,m}^{j+1}
\]
is equivalent to
\[
 \rho_{k,m}\geq \frac{j+1}{k-j}.
\]
Having this for all $j=0,\dots,m-1$ is equivalent to having it for $j=m-1$:
\[
 \rho_{k,m}\geq \frac{m}{k-m+1},
\]
which is immediate from the definition of $\rho_{k,m}$.

After these reductions, it will be enough to prove the following statement.

\emph{Claim:} There exist $\epsilon>0$ and $k_0,m_0\in\N$ such that for $k\geq k_0$, $m\geq m_0$ and for $j= [m-\sqrt m]$,
\[
 \frac 1{(1+\rho_{k,m})^k} 
 \binom{k}{j} \rho_{k,m}^j \geq \frac {\epsilon }{\sqrt m}.
\]

\medskip
Once this is proved we shall have the above inequality for $j\geq [m-\sqrt m]$, hence
\[
 \frac 1{(1+\rho_{k,m})^k}\sum_{j<m} \binom{k}{j}\rho_{k,m}^j\geq \frac 1{(1+\rho_{k,m})^k}\sum\limits_{[m-\sqrt m]<j<m} \binom{k}{j} \rho_{k,m}^j\geq \sqrt m\, \frac {\epsilon}{\sqrt m}=\epsilon,
\]
as desired.

The proof of the Claim is just another computation. Let us denote momentarily $j=[m-\sqrt m]$. Using the definition of $\rho_{k,m}$, the inequality we want to prove is
\[
 B_{k,m}:=\Bigl(\frac{k-m}k\Bigr)^k\frac {k!}{(k-j)! j!}\, \Bigl(\frac m{k-m}\Bigr)^j \geq \frac {\epsilon }{\sqrt m}.
\]
By Lemma~\ref{lemma:0} and Stirling's formula
\begin{align*}
 B_{k,m}:&=\Bigl(\frac{k-m}k\Bigr)^{k-j} \frac{k!}{k^j (k-j)!}\, \frac {m^j}{j!}\geq 
 \Bigl(\frac{k-m}k\Bigr)^{k-j} e^{-j} \Bigl(\frac{k-j}k\Bigr)^{-(k-j)} \frac {m^j}{j!}\\
 &\simeq
 \Bigl(\frac{k-m}{k-j}\Bigr)^{k-j}\, \frac 1{\sqrt{2\pi j}}\, \frac {m^j}{j^j}.
\end{align*}
Recalling that $j=[m-\sqrt m]$ we have finally
\begin{align*}
 B_{k,m}&\succsim \Bigl(1-\frac{\sqrt m}{k-m+\sqrt m}\Bigr)^{k-m+\sqrt m} \Bigl(\frac{m}{m-\sqrt m}\Bigr)^{m-\sqrt m} \frac 1{\sqrt{2\pi m-\sqrt m}}\\
 &\simeq \left[\Bigl(1-\frac{\sqrt m}{k-m+\sqrt m}\Bigr)^{\frac{k-m+\sqrt m}{\sqrt m}}\right]^{\sqrt m} 
 \left[\Bigl(1+\frac{\sqrt m}{m-\sqrt m}\Bigr)^{\frac{m-\sqrt m}{\sqrt m}}\right]^{\sqrt m} \frac 1{\sqrt{2\pi(1-\frac 1{\sqrt m})}}\frac 1{\sqrt m}\\
 &\simeq e^{-1}\, e\, \frac 1{\sqrt m}.
\end{align*}
\end{proof}

\subsection{Final remark} Since dim$\, \P_k=k+1$, an obvious necessary condition for $X=\bigl\{(\Lambda_k, m^{(k)})\bigr\}_{k\geq 1}$ to be a zero-array for $\P_k$ is that $\sum_{\lambda\in\Lambda_k}m_\lambda\leq k$ for all $k\geq 1$. This is equivalent to
\[
\sum_{\lambda\in\Lambda_k} \nu\bigl(D_{\lambda,m_\lambda}\bigr)=\sum_{\lambda\in\Lambda_k}\frac{m_\lambda}k\leq 1.
\]
A more careful estimate shows that actually, for $k$ big enough,
\[
 \nu\bigl(\C\setminus\mathcal U_k\bigr)\succsim \frac 1k,
\]
where, as before, $\mathcal U_k=\cup_{\lambda\in\Lambda_k} D_{\lambda,m_\lambda}$. This implies the analogue of Theorem 1.8 in \cite{BHKM}: if
\[
 \lim_{k\to\infty} k\nu\bigl(\C\setminus \mathcal U_k\bigr)=0,
\]
then $X$ cannot be a zero array of $\P_k$. This can be seen as a weaker version of the geometric condition in Theorem~\ref{thm:L2-sampling-suff}.

It is not difficult to see the above estimate. On the one hand, since the radii of the disks $D_{\lambda,m_\lambda}$ are at least $1/\sqrt k$, the result is clear in case there are no three such discs intersecting. 
If there exist three intersecting disks $D_{\lambda_j,m_{\lambda_j}}$, $\lambda_j\in\Lambda_k$, $j=1,2,3$, the analogue of Lemma 2.1 in \cite{BHKM} shows that there are $c>0$ and at least two disks, say $D_{\lambda_j,m_{\lambda_j}}$, $j=1,2$, such that
\[
 \nu\bigl(D_{\lambda_1,m_{\lambda_1}}\cap D_{\lambda_2,m_{\lambda_2}}\bigr)\geq \frac ck.
\]
Then 
\begin{align*}
 \nu(\mathcal U_k)&\leq \nu\bigl(D_{\lambda_1,m_{\lambda_1}}\cup D_{\lambda_2,m_{\lambda_2}}\bigr)+\sum_{\substack{\lambda_j\in\Lambda_k \\ j\neq 1,2}} \frac{m_{\lambda_j}}k\\
 &=\frac{m_{\lambda_1}}k+\frac{m_{\lambda_2}}k-\nu\bigl(D_{\lambda_1,m_{\lambda_1}}\cap D_{\lambda_2,m_{\lambda_2}}\bigr)+\sum_{\substack{\lambda_j\in\Lambda_k \\ j\neq 1,2}} \frac{m_{\lambda_j}}k\\
 &\leq\sum_{\lambda\in\Lambda_k}\frac {m_\lambda}k -\frac ck\leq 1-\frac ck,
\end{align*}
and therefore
\[
 \nu\bigl(\C\setminus\mathcal U_k\bigr)=\nu(\C)-\nu(\mathcal U_k)=1-\nu(\mathcal U_k)\geq \frac ck.
\]

\end{document}